\documentclass{amsart}
\usepackage{amssymb}
\usepackage{stmaryrd} 
\usepackage{amsmath} 
\usepackage{amscd}
\usepackage{amsbsy}
\usepackage{commath}
\usepackage{comment, enumerate}
\usepackage{geometry}
\usepackage[matrix,arrow]{xy}
\usepackage{hyperref}
\usepackage{tabularx}

\usepackage{mathrsfs}
\usepackage{color}
\usepackage{mathtools,caption}
\usepackage{tikz-cd}
\usepackage{longtable}
\usepackage[utf8]{inputenc}
\usepackage[OT2,T1]{fontenc}
\usepackage[normalem]{ulem}
\usepackage{hyperref}
\hypersetup{
	colorlinks=true,
	linkcolor=blue,
	filecolor=violet,
	citecolor=orange,
	urlcolor=purple,
	pdftitle={actapaper - v2},
}

\newtheorem{thm}{Theorem}[section]
\newtheorem*{thm*}{Theorem}
\newtheorem{dfn}[thm]{Definition}
\newtheorem{lem}[thm]{Lemma}
\newtheorem{prop}[thm]{Proposition}
\newtheorem{cor}[thm]{Corollary}

\newtheorem{claim}[thm]{Claim}

\theoremstyle{definition}
\newtheorem{example}{Example}
\newtheorem{assumption}{Assumption}

\def\Z{\mathbf{Z}}

\def\Q{\mathbf{Q}}

\def\N{\mathbf{N}}

\def\O{\mathcal{O}}

\def\a{\alpha}

\def\p{\mathfrak{p}}

\def\Gal{\text{Gal}}
\def\Cl{\text{Cl}}

\def\Am{\text{Am}}

\newcommand{\AC}[1]{{\color{black}#1}}  

\DeclareFontFamily{U}{wncy}{}
\DeclareFontShape{U}{wncy}{m}{n}{<->wncyr10}{}
\DeclareSymbolFont{mcy}{U}{wncy}{m}{n}
\DeclareMathSymbol{\Sha}{\mathord}{mcy}{"58}

\def\sse{\subseteq}
\title{The non-$p$-part of the fine Selmer group in a $\Z_p$-extension}

\author[]{Adithya Chakravarthy}
\date{}
\begin{document}
	
	\maketitle
	
	\begin{abstract}
		
		Fix two distinct primes $p$ and $\ell$. Let $A$ be an abelian variety over $\Q(\zeta_{\ell})$, the cyclotomic field of $\ell$-th roots of unity. Suppose that $A(\Q(\zeta_{\ell}))[\ell] \neq 0$. We show that there exists a number field $L$ and a $\Z_p$ extension $L_{\infty}/L$ where the $\ell$-primary fine Selmer group of $A$ grows arbitrarily quickly. This is a fine Selmer group analogue of a theorem of Washington that there are certain (non-cyclotomic) $\Z_p$-extensions where the $\ell$-part of the class group can grow arbitrarily quickly. We also prove this for a wide class of non-commutative $p$-adic Lie extensions. Finally, we include several examples to illustrate this theorem.
	\end{abstract}
	
	\section{Introduction}
	
	\label{sec:intro}
	
	We begin with a fundamental theorem of Iwasawa, which serves as the starting point of Iwasawa theory.  Let $K$ be a number field and let $K_{\infty}/K$ be a $\Z_p$-extension: a Galois extension with Galois group isomorphic to the additive group $\Z_p$ of $p$-adic integers. For such an extension $K_{\infty}/K$, there exists a unique sequence of fields 
	\[ K = K_0 \sse K_1 \sse \dots \sse K_n \sse \dots \sse K_{\infty} \]
	such that each $K_n/K$ is a cyclic extension of degree $p^n$. Iwasawa \cite{iwasawa} proved the following now famous theorem about the growth of class numbers in such towers.	 
	\begin{thm*}[Iwasawa]
		Let $K$ be a number field and let $K_{\infty}/K$ be a $\Z_p$ extension with layers $K_n$.  Suppose that $p^{e^n}$ is the exact power of $p$ dividing the class number of $K_n$. Then there exist integers $\mu, \lambda, \nu$  such that 
		\[ e_n = \mu p^n + \lambda n + \nu \]
		for all sufficiently large values of $n$.
	\end{thm*}	
	A large part of classical Iwasawa theory is devoted to studying the invariants $\mu$ and $\lambda$ in the above formula. In a beautiful paper, Iwasawa showed \cite[Theorem 1]{iwasawa-mu-large} that there are $\Z_p$-extensions for which the $\mu$-invariant can be \textit{arbitrarily} large.
	
	\begin{thm}[Iwasawa]
		\label{iwasawa-mu-large}
		Let $N \geq 1$. There exists a number field $L$ and a $\Z_p$-extension $L_{\infty}/L$ such that $\mu \geq N$.
	\end{thm} 
	
	Now that we have discussed the $p$-part of the class group in a $\Z_p$-extension, we now discuss the $\ell$-part of the class group in a $\Z_p$-extension, where $\ell \neq p$ are \textit{distinct} primes. The fundamental theorem in this area is due to Washington \cite{washington-l-neq-p}:
	
	\begin{thm}[Washington]
		\label{washington: l_neq_p_bounded}
		Let $\ell \neq p$ be distinct primes Let $K$ be an abelian extension of the field $\Q$ of rational numbers. Let $K^{\text{cyc}}/K$ be the cyclotomic $\Z_p$ extension of $K$ and let $\ell^{e_n}$ be the exact power of $\ell$ dividing the class number of $K_n$. Then $e_n$ is bounded as $n \to \infty$. 
	\end{thm}
	
	Based on this, one might reasonably guess that the $\ell$-part of the class group is bounded in an \textit{arbitrary} $\Z_p$-extension. But this turns out to be false, as proven in \cite[Theorem 6]{washington-annalen}.
	
	\begin{thm}[Washington]
		\label{washington: l_neq_p_unbounded}
		Let $N \geq 1$. There exists a number field $L$ and a $\Z_p$-extension $L_{\infty}/L$ such that 
		\[ e_n \geq Np^n, \] 
		where $\ell^{e_n}$ is the exact power of $\ell$ dividing the class number of $L_n$.
	\end{thm}
	
	The purpose of this article is to discuss analogues of Theorems~\ref{iwasawa-mu-large} and \ref{washington: l_neq_p_unbounded} for fine Selmer groups of elliptic curves. Let $E$ be an elliptic curve over a number field $F$. The Mordell-Weil Theorem says the group $E(F)$ of rational points is a finitely generated abelian group. This arithmetic of this group is essentially controlled by the Selmer group of $E/F$. In \cite{mazur}, Mazur introduced the Iwasawa theory of Selmer groups in $\Z_p$-extensions of $F$. The notion of the Fine Selmer group was formally introduced by Coates and Sujatha in \cite{coates-sujatha}, \AC{even though it had been studied by Rubin \cite{rubin} and Perrin-Riou \cite{perrin-riou1, perrin-riou2} under various guises in the late 80's and early 90's.} In \cite{coates-sujatha}, Coates and Sujatha showed that these fine Selmer groups have stronger finiteness properties than classical Selmer groups.
	
	In \cite[Theorem 4.2]{debanjana-fsg}, Kundu proved an analogue of Theorem ~\ref{iwasawa-mu-large} for fine Selmer groups. Kundu proved that if $A$ is an abelian variety and $N \geq 1$ is an integer, then there exists a number field $L$ and a $\Z_p$-extension $L_{\infty}/L$ such that the $\mu$-invariant of the fine Selmer group of $A$ over $L_{\infty}$ is at least $N$. In other words, $\mu$-invariants of fine Selmer groups can be arbitrarily large. \AC{And in \cite[Theorem B]{kundu-lei}, Kundu and Lei proved a Fine Selmer group analogue of Theorem~\ref{washington: l_neq_p_bounded}: in the \textit{cyclotomic} $\Z_p$-extension, the $\ell$-part of the Fine Selmer group of $A$ stabilizes.} 
	
	The purpose of this paper is to prove a fine Selmer group analogue of Theorem~\ref{washington: l_neq_p_unbounded}. If $\ell$ is a prime and $A$ is an abelian variety over a number field $F$, let $R_{{\ell}^{\infty}}(A/F)$ denote the $\ell$-primary fine Selmer group of $A$ over $F$. (See Section~\ref{fsg:preliminaries} for the precise definition.) If $G$ is an abelian group, the \textit{$\ell$-rank} of $G$ is defined by $r_{\ell}(G) = \dim_{\Z/{\ell}\Z} G[\ell]$. Note that it is possible to have $r_{\ell}(G)=\infty$; this is true for example if $G$ contains infinitely many copies of $\Z/\ell\Z$. (See \cite[Section 3]{lim-murty} as a reference for this definition.) Here is our main result:
	
	\begin{thm}
		\label{main-thm-zp}
		Let $\ell \neq p$ be distinct primes. Let $\Q(\zeta_{\ell})$ be the cyclotomic field of $\ell$-th roots of unity and let $A$ be an abelian variety over $\Q(\zeta_{\ell})$. Suppose that $A(\Q(\zeta_{\ell}))[\ell] \neq 0$. For every integer $N \geq 1$, there exists a finite extension $L/\Q(\zeta_{\ell})$ and a $\Z_p$-extension $L_{\infty}/L$ such that
		\[ r_{\ell}(R_{{\ell}^{\infty}}(A/L_n)) \geq Nq^n  \]
		for all $n \geq 0$, where $q = \min\{\ell, p\}$. In particular, $r_{\ell}(R_{{\ell}^{\infty}}(A/L_n)) \to \infty$ as $n \to \infty$. 
	\end{thm}

	More generally, we can also prove this theorem for many non-commutative $p$-adic Lie-extensions. The statement of our theorem requires a definition:  
		
		\begin{dfn}
			A pro-$p$ group $\Gamma$ is \textit{uniform of dimension $d$} if it is topologically finitely generated by $d$ generators and there exists a unique filtration by the $p$-descending central series of $\Gamma$. In other words, we have 
			\[ \Gamma = \Gamma_0 \supset \Gamma_1 \supset \dots \supset \Gamma_n \supset \dots \]
			such that each $\Gamma_{n+1}$ is normal in $\Gamma_n$ and $\Gamma_n / \Gamma_{n+1} \simeq (\Z/p\Z)^d$.  
		\end{dfn}
		
	If $F$ is a number field and $\mathfrak{p}$ is a prime ideal of $F$, then the $\mathfrak{p}$-class group of $F$ is the quotient of $\text{Cl}(F)$ by the subgroup generated by the ideal class of $\mathfrak{p}$.		
		\begin{assumption}
			\label{assumption: m>2}
			Let $\Gamma$ be a uniform pro-$p$ group with a fixed-point-free automorphism of order $m$, where $m > 2$ is a prime different from $p$. Assume that:
				\begin{enumerate}
				\item there exists a $\Z/m\Z$ extension of number fields $F/F_0$, where $F_0$ is totally imaginary,  
				\item the field $F$ contains the $p$-th roots of unity,
				\item there is a unique prime $\mathfrak{p}$ of $F$ lying over the rational prime $p$, and 
				\item the $p$-part of the $\mathfrak{p}$-class group of $F$ is trivial. 
			\end{enumerate}
		\end{assumption}
		
		\begin{thm} 
			\label{main-thm-gamma}		
			Let $\ell \neq p$ be distinct primes. Let $\Q(\zeta_{\ell})$ be the cyclotomic field of $\ell$-th roots of unity and let $A$ be an abelian variety over $\Q(\zeta_{\ell})$. Suppose that $A(\Q(\zeta_{\ell}))[\ell] \neq 0$. Let $\Gamma$ be a uniform pro-$p$ group with a fixed-point-free automorphism of order $m$. If $m > 2$, assume that Assumption~\ref{assumption: m>2} holds. 
			
			Then for every integer $N \geq 1$, there exists a finite extension $L/\Q(\zeta_{\ell})$ and a $\Gamma$-extension $L_{\infty}/L$ such that
			\[ r_{\ell}(R_{{\ell}^{\infty}}(A/L_n)) \geq Nq^n  \]
			for all $n \geq 0$, where $q = \min\{\ell, p\}$. In particular, $r_{\ell}(R_{{\ell}^{\infty}}(A/L_n)) \to \infty$ as $n \to \infty$. 
		\end{thm}
		
		Let $\text{Sel}_{{\ell}^{\infty}}(A/F)$ denote the $\ell$-primary (usual) Selmer group of $A$ over $F$. (See Section~\ref{fsg:preliminaries} for the precise definition.) Since the fine Selmer group is a subgroup of the usual Selmer group, we have the following
		
		\begin{cor}
			Retain the notations and assumptions of Theorem~\ref{main-thm-gamma}. For every integer $N \geq 1$, there exists a finite extension $L/\Q(\zeta_{\ell})$ and a $\Gamma$-extension $L_{\infty}/L$ such that
			\[ r_{\ell}(\text{Sel}_{{\ell}^{\infty}}(A/L_n)) \geq Nq^n  \]
			for all $n \geq 0$, where $q = \min\{\ell, p\}$. In particular, $r_{\ell}(\text{Sel}_{{\ell}^{\infty}}(A/L_n)) \to \infty$ as $n \to \infty$.
		\end{cor}
	
		We can give several numerical examples at the end of the paper to show that Assumption~\ref{assumption: m>2} often holds.
	
	\textit{Strategy.} \AC{We follow Iwasawa's approach in \cite{iwasawa-mu-large}, which inspired Washington's approach in \cite[Section VI]{washington-annalen}.} We construct a $\Z_p^d$-extension $L_{\infty}/L$ where the $\ell$-rank of the \textit{class group} is unbounded. 
	
	\begin{thm} 
	\label{thm-class-groups}
		Let $\ell \neq p$ be distinct primes. Let $\Gamma$ be a uniform pro-$p$ group. If $m>2$, then assume Assumption~\ref{assumption: m>2}. For every integer $N \geq 1$, there exists a finite extension $L/\Q(\zeta_{\ell})$ and a $\Gamma$-extension $L_{\infty}/L$ such that for all $n \geq 0$:
		\[ r_{\ell}(\Cl(L_n)) \geq Np^n.  \]
		In particular, $r_{\ell}(\Cl(L_n)) \to \infty$ as $n \to \infty$.
	\end{thm}
	
	We then use results of Lim-Murty \cite{lim-murty} to show that the $\ell$-rank of the fine Selmer group is close in size to the $\ell$-rank of the class group. Putting these two results together, we conclude that the $\ell$-part of the fine Selmer group is unbounded in $L_{\infty}/L$, proving Theorem~\ref{main-thm-gamma}. 

    \bigskip
    \textbf{Acknowledgements.} I want to thank Debanjana Kundu for guiding me through my first steps in Iwasawa theory, and also for suggesting this problem to me. This paper would not have been possible without your support, so thank you! I would also like to thank Professor Kumar Murty and members of the GANITA lab for patiently listening to me present the proofs in this paper. Finally, I thank the anonymous referee for their comments which greatly improved the paper's content and exposition.  	
	\section{Construction of the \texorpdfstring{$\Gamma$}{G}-extension \texorpdfstring{$K_{\infty}/K$}{Kinfty}} 
	
	In this section we will construct a $\Gamma$-extension $K_{\infty}/K$ where infinitely many primes split completely. Here $\Gamma$ will be a uniform pro-$p$ group with a fixed-point-free automorphism $\tau$ of order $m$. The construction proceeds differently in the cases $m = 2$ and $m > 2$.
	
	\subsection{The case $m=2$}
	
	By \cite[Theorem 2.17]{group-theory}, if $\Gamma$ is a uniform pro-$p$ group with an fixed-point-free automorphism $\tau$ of order $m=2$, then $\Gamma \simeq \Z_p^d$ for some $d \geq 1$.
	
	To motivate the next proposition, recall that if $K$ is an imaginary quadratic field, then there is a $\Z_p$-extension $K_{\infty}/K$ called the anticyclotomic $\Z_p$-extension of $K$. This extension has the property that infinitely many primes of $K$ split completely in $K_{\infty}$. The next proposition generalizes this to $\Z_p^d$ extensions for $d \geq 1$. 
	
	\begin{prop}
		\label{prop: part0-zpd-extensions}
		Let $K$ be a CM field such that $K/\Q$ is a $\Z/2d\Z$-extension. Suppose that there is only prime in $K$ lying over $p$. Then there is a $\Z_p^d$-extension $K_{\infty}/K$ such that infinitely many primes of $K$ split completely in $K_{\infty}$.
	\end{prop}

	\begin{proof} 
		This fact is well-known but for lack of a reference, we sketch the proof here. The below construction is reproduced from \cite[Section 2]{longo}. Let $K^+$ be the maximal totally real subfield of $K$. For any integral ideal $\mathfrak{c} \sse \O_{K^+}$, let $\O_{\mathfrak{c}} = \O_{K^+} + \mathfrak{c}\O_K$ be the order of conductor $\mathfrak{c}$ in $K$. The \textit{ring class field} $K[\mathfrak{c}]/K$ of $K$ of conductor $\mathfrak{c}$ is the Galois extension of $K$ such that there an isomorphism via the Artin map: 
		\[ \Gal(K[c]/K) \simeq \Cl(\O_{\mathfrak{c}}). \] 
		Let $\p$ be the unique prime of $K^+$ lying over $p$. Put $K[\p^{\infty}] = \cup_{n=1}^{\infty} K[\p^n]$. Define $K_{\infty}$ to be the unique subfield of $K[\p^{\infty}]$ satisfying 
		\[ \Gal(K_{\infty}/K) \simeq \Z_p^{[K^+_{\p}:\Q_p]} = \Z_p^{d}. \]
		Suppose that $\ell$ is a rational prime which is inert in $K$. Then the ideal class of $\ell$ is trivial in $\Cl(K)$ and hence class field theory gives that $\ell$ splits completely in any ring class field $K[\mathfrak{c}]$ of conductor coprime to $\ell$. (See \cite[Section 2.6.3]{nekovar}.) In particular, if $\ell$ is inert in $K/\Q$ and $\ell$ is coprime to $p$ then $\ell$ splits completely in $K_{\infty}$. By the Chebotarev density theorem, there are infinitely many such primes.
		
	\end{proof}
	
	\subsection{The case $m > 2$}
	Here is the main result: 
	\begin{prop}
		\label{prop: part0-gamma-extensions}
		Let $\Gamma$ be a uniform pro-$p$ group. Assume the Assumption~\ref{assumption: m>2}.	Then there exists a Galois extension $K/F$ and a $\Gamma$-extension $K_{\infty}/K$ such that infinitely many primes of $K$ split completely in $K_{\infty}$. 
	\end{prop}

	Now let $F$ be a number field and let $F_{\text{max,}p}$ be the maximal pro-$p$ extension of $F$ unramified outside the primes above $p$. The number field $F$ is called \textit{$p$-rational} if $\Gal(F_{\text{max,}p}/F)$ is pro-$p$ free. 
	
	\begin{lem}
		\label{lem: p-rational}
		Let $F$ be a number field with a primitive $p$-th root of unity. Then $F$ is $p$-rational if and only if there exists a unique prime $\mathfrak{p}$ above $p$ and the $p$-part of the $\mathfrak{p}$-class group of $F$ is trivial. 
	\end{lem}
	
	\begin{proof}
		This is \cite[Theorem IV.3.5]{gras}.
	\end{proof}
	
	\begin{lem}
		\label{lem: hajir-maire}
		Keep the notations and assumptions from Proposition~\ref{prop: part0-gamma-extensions}. Let $n$ be an integer such that $[F_0: \Q] p^n \geq 2d$ and let $K_0$ (resp. $K$) be the $n$-th layer of the cyclotomic $\Z_p$-extension of $F_0$ (resp. $F$). Then there exists an intermediate field $K \subset K_{\infty} \subset K_{\text{max, }p}$ such that $K_{\infty}$ is Galois over $K_0$ with Galois group $\Gamma \rtimes \langle \tau \rangle$. Suppose $\tau$ acts fixed-point-freely on $\Gamma$. Then every place of $K_0$ which is inert in $K/K_0$ splits completely in $K_{\infty}/K$.  
	\end{lem}
	
	\begin{proof}
		This follows from \cite[Proposition 3.6]{hajir-maire} and \cite[Proposition 3.7]{hajir-maire} if $F$ is $p$-rational. But Lemma~\ref{lem: p-rational} gives conditions for $F$ to be $p$-rational and $F$ satisfies those conditions. 
	\end{proof}

	Proposition~\ref{prop: part0-gamma-extensions} now follows from Lemma~\ref{lem: hajir-maire}, because by the Chebotarev density theorem there are infinitely many primes that are inert in the cyclic extension $K/K_0$.  
	
	\section{Construction of the $\Gamma$-extension $L_{\infty}/L$}
	
	Fix $N \geq 1$. In this section we will construct a $\Gamma$-extension $L_{\infty}/L$ with the properties in Theorem~\ref{thm-class-groups}, i.e: such that 
	\[ r_{\ell}(\Cl(L_n)) \geq Np^n\]
	for all $n \geq 0$.  
	
	\begin{prop}
		\label{prop: part1}
		Let $N \geq 1$ be an integer. Let $\Gamma$ be a uniform pro-$p$ group of dimension $d$ with a fixed-point-free automorphism of order $m$. If $m>2$, assume Assumption~\ref{assumption: m>2}.  Let $K_{\infty}/K$ denote the $\Gamma$-extension from Proposition~\ref{prop: part0-zpd-extensions} (resp. Proposition~\ref{prop: part0-gamma-extensions}) if $m=2$ (resp. $m>2$). There exists a finite extension $L/K$ and a $\Gamma$-extension $L_{\infty}/L$ satisfying the following:
		\begin{enumerate}
			\item \label{properties-of-Linfty-1}  The extension $L_{\infty}$ contains $K_{\infty}$. Furthermore, for all $n \geq 0$, the number of primes ramifying in $L_n/K_n$ is at least 
			\[ \left(N + md\ell (\ell-1) \right) p^n. \] 
			\item \label{properties-of-Linfty-2} We have $[L_n:\Q] \geq m\ell(\ell-1)dp^n$ for all $n\geq 0$.  
		\end{enumerate}
	\end{prop}
	
	\begin{proof}	
		We only prove the case $m>2$; the proof of $m=2$ is identical and left to the reader. By Proposition~\ref{prop: part0-gamma-extensions}, infinitely many primes of $K$ split completely in $K_{\infty}/K$. Let $t \geq 1$ be an integer (to be chosen later) and primes $v_1, \dots, v_t$ in $K$ that split completely in $K_{\infty}/K$.  
		
		\begin{claim}
			There exists $\a \in K$ such that $\text{ord}_{v_i}(\a) = 1$ for all $i=1, \dots , t$. 
		\end{claim}
		
		\begin{proof}[Proof of Claim]
			To see this, consider the ideal class $I = [v_1 \cdot \dots \cdot v_t] \in \Cl(K)$. Then $II^{-1}$ is trivial in $\Cl(K)$ so there exists a fractional ideal $w$ of $K$ such that $v_1 \cdot \dots \cdot v_t \cdot w$ is a principal ideal; write $v_1 \cdot \dots \cdot v_t \cdot w = (\a)$ for some $\a \in K$. Then $\text{ord}_{v_i} (\a) \geq 1$ for all $i=1, \dots t$. We can ensure that $\text{ord}_{v_i} (\a)$ is \textit{exactly} one by dividing $w$ through by $v_i$ if necessary. This proves the claim.
		\end{proof}
		
		Now let $\a \in K$ be such that $\text{ord}_{v_i}(\a) = 1$ for all $i=1, \dots , t$. Put $L = K(\a^{1/\ell}, \zeta_{\ell})$. Then $L/K(\zeta_{\ell})$ cyclic degree $\ell$ extension where $v_1, \dots, v_t$ ramify. Put $L_{\infty} = K_{\infty} L$. Then $L_{\infty}/L$ is a $\Gamma$-extension. We summarize this in a diagram:
		
		
		
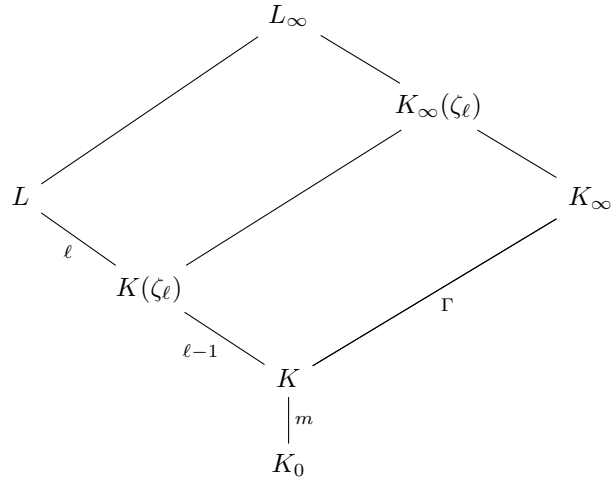
\begin{center}
			\begin{tikzcd}
				&& {L_{\infty}} \\
				&&& {K_{\infty}(\zeta_{\ell})} \\
				L &&&& {K_{\infty}} \\
				& K(\zeta_{\ell}) \\
				&& K && {} \\
				&& {K_0} && {}
				\arrow[no head, from=3-1, to=1-3]
				\arrow[no head, from=4-2, to=2-4]
				\arrow["\ell"', no head, from=3-1, to=4-2]
				\arrow["{\ell-1}"', no head, from=4-2, to=5-3]
				\arrow[no head, from=2-4, to=3-5]
				\arrow[no head, from=1-3, to=2-4]
				\arrow["m", no head, from=5-3, to=6-3]
				\arrow[no head, from=5-3, to=3-5]
				\arrow["{\Gamma}"', no head, from=5-3, to=3-5]
			\end{tikzcd}
		\end{center}

		Let $K_n$ (resp. $L_n$) be the $n$-th layer of the $\Gamma$-extension $K_{\infty}$ (resp. $L_{\infty}$). The primes $v_1, \dots, v_t$ ramify in $L/K$. Furthermore, all the primes of $K_n$ lying over $v_1, \dots, v_t$ must ramify in $L_n$ as well. Since each $v_i$ splits completely, there are $tp^n$ such primes of $K_n$. Therefore, the number of primes of $L_n/K_n$ that ramify is at least $tp^n$. Now set 
		\[ t \coloneqq N + md\ell (\ell-1). \]
		This proves Property (\ref{properties-of-Linfty-1}). 
		
		To prove Property (\ref{properties-of-Linfty-2}), we just count degrees in the above field diagram, noting that $[L_n:L] = dp^n$. This completes the proof.
	\end{proof}

	\section{Growth of class groups in \texorpdfstring{$L_{\infty}/L$}{}: the proof of Theorem~\ref{thm-class-groups}}
	
	We want to show that the $\ell$-part of the class group in the $\Gamma$-extension $L_{\infty}/L$ is unbounded. Our main tool to do this is the so-called \textit{ambiguous class number formula}. 
	
	\begin{dfn}
		Let $\ell$ be a prime. Let $K$ be a number field and $L/K$ be a cyclic $\Z/\ell\Z$-extension. Let $\sigma \in \Gal(L/K)$ be a generator. An ideal class $[\mathfrak{a}] \in \Cl(L)$ is called \textbf{strongly ambiguous} if $[\mathfrak{a}]^{\sigma-1} = (1)$. 
	\end{dfn}
	The subgroup of $\Cl(L)$ consisting of strongly ambiguous classes is denoted by $\Am_{\text{st}}(L/K)$.
	
	The following is given in \cite[Proposition 4.5]{debanjana-fsg}.
	\begin{prop}[Ambiguous Class Number Formula]
		\label{ambiguous-cnf}
		Let $\ell$ be a prime. Let $K$ be a number field and $L/K$ be a cyclic $\Z/\ell\Z$-extension with $\sigma$ a generator of the Galois group $\Gal(L/K)$. Then 
		\[r_{\ell}\left( \Am_{\text{st}}(L/K) \right) \geq T - [L:\Q], \]
		where $T$ is the number of ramified primes in $L/K$.
	\end{prop}
	
	\begin{proof}[Proof of Theorem~\ref{thm-class-groups}]
		
		Observe that for each $n \geq 1$, the extension $L_n/K_n$ is cyclic of degree $\ell$. Applying Proposition~\ref{ambiguous-cnf} to $L_n/K_n$, we have
		\begin{align*}
			r_{\ell}\left( \Am_{\text{st}}(L_n/K_n) \right) 
			&\geq T - [L_n:\Q],
		\end{align*}
		the number of primes ramifying in $L_n/K_n$ is at least 
		\[ \left(N + 2d\ell (\ell-1) \right) p^n. \] 
		We have $[L_n:\Q] \geq m\ell(\ell-1)dp^n$ for all $n\geq 0$.  
		
		By Property (\ref{properties-of-Linfty-2}), we have $T \geq \left( N + md\ell(\ell-1) \right)p^n$ and $[L_n:\Q] \geq  md\ell(\ell-1)p^n$. And since $\Am_{\text{st}}(L_n/K_n)$ is a subgroup of $\Cl(L_n)$, we have 
		\[ r_{\ell} \left( \Cl(L_n)\right) \geq r_{\ell}\left( \Am_{\text{st}}(L_n/K_n) \right). \]
		Combining these, we obtain:
		\begin{align*}
			r_{\ell} \left( \Cl(L_n)\right)
			&\geq r_{\ell}\left( \Am_{\text{st}} (L_n/K_n) \right) \\
			&\geq T-[L_n:\Q] \\
			&\geq \left( N + md\ell(\ell-1) \right)p^n - md\ell(\ell-1)p^n \\
			&= Np^n
		\end{align*}
		for all $n \geq 0$. This completes the proof of Theorem~\ref{thm-class-groups}.
	\end{proof}
	
	\section{Application to fine Selmer groups: the proof of Theorems~\ref{main-thm-gamma}}
	
	\subsection{Review of fine Selmer Group}
	\label{fsg:preliminaries}
	
	Let $F$ be a number field and $p$ a prime. Let $A$ be an abelian variety over $F$ and let $S$ be a finite set of primes containing $S_p \cup S_{\text{bad}} \cup S_{\infty}$. Denote by $F_S$ the maximal extension of $F$ unramified outside $S$.
	
	The usual $p^{\infty}$-Selmer group of $A$ is defined by 
	\[ \text{Sel}_{p^{\infty}}(A/F) = \ker \left( H^1(F, A[p^{\infty}]) \to \prod_v H^1(F_v, A)[p^{\infty}] \right). \]
	Here $v$ runs through all the primes of $F$. The fine Selmer group of $A$ is defined by the exact sequence 
	\[ 0 \to R_{p^{\infty}}(A/F) \to \text{Sel}_{p^{\infty}}(A/F) \to \bigoplus_{v \vert p} A(F_v) \otimes_{\Z} \Q_p/\Z_p.   \]
	
	\subsection{Proof of Theorem~\ref{main-thm-gamma}}
	
	Let $F$ be a number field and $S$ a finite set of places of $F$. \AC{The $S$-class group of $F$, denoted $\text{Cl}_S(F)$, is the quotient of $\text{Cl}(F)$ by the subgroup generated by the ideal classes of prime ideals in $S$.} The following proposition is proven in \cite[Lemma 4.3]{lim-murty}.
	\begin{prop}
		\label{prop: lim-murty}
		Let $A$ be an dimensional Abelian variety over a number field $F$. Let $S$ be a finite set of primes containing $S_{\ell} \cup S_{\text{bad}} \cup S_{\infty}$. Suppose that $A(F)[\ell] \neq 0$. Then
		\[ r_{\ell}\left( R_{\ell^{\infty}} (A/F) \right) \geq r_{\ell}\left(\text{Cl}_S(F) \right) \cdot r_{\ell}\left(A(F)[\ell] \right) - 2 \dim(A),  \]
		where $\dim(A)$ denotes the dimension of the Abelian variety $A$.
	\end{prop}
	
	We will first relate the $S$-class group of $F$ to the class group of $F$.
	
	\begin{lem}
		\label{lem: S-class-group}
		Let $L$ be a number field and let $\ell$ be a rational prime. Let $S$ be a finite set of places of $L$ containing the primes above $\ell$. Let $s_0$ be the number of finite primes in $S$. Then for all $n$: 
		\[ \left| r_{\ell} \left(\Cl(L_n) \right) - r_{\ell}  \left(\text{Cl}_S(L_n) \right)  \right| \geq 2s_0\ell^n. \]
	\end{lem}
	
	\begin{proof}
		We have reproduced this proof from \cite[Lemma 4.6, Step A]{debanjana-fsg}. 
		Let $s_n$ be the number of finite primes of $L_n$ lying over a prime of $S$. Consider the following short exact sequence for all $n$ \cite[Lemma 10.3.12]{neukirch},
		\[ \Z^{s_0} \to \Cl(L_n) \to \text{Cl}_S(L_n).  \]
		Taking $\ell$-ranks of this sequence, we obtain \cite[Lemma 3.2]{lim-murty}:
		\[  \left| r_{\ell} \left(\Cl(L_n) \right) - r_{\ell}  \left(\text{Cl}_S(L_n) \right)  \right| \geq 2s_0\ell^n. \]
	\end{proof}
	
	\begin{proof}[Proof of Theorem~\ref{main-thm-gamma}]
		Recall that we have two distinct primes $\ell \neq p$ and an abelian variety $A$ defined over $\Q(\zeta_{\ell})$.  Let $S = S_p \cup S_{\text{bad}} \cup S_{\infty}$. Let $s_0$ be the number of finite places of $S$. We can apply Theorem~\ref{thm-class-groups} \AC{(replacing $N$ with $N + 2s_0$)} to construct a $\Gamma$-extension $L_{\infty}/L$ such that 
		\[ r_{\ell}(\Cl(L_n)) \geq (N+2s_0)p^n \]
		for all $n \geq 0$.  Proposition~\ref{prop: lim-murty} tells us that for all $n\geq 0$:
		\[ r_{\ell}\left( R_{\ell^{\infty}} (A/L_n) \right) \geq r_{\ell}\left(\text{Cl}_S(L_n) \right) \cdot r_{\ell}\left(A(L_n)[\ell] \right) - 2\dim(A). \]
		Since we assumed $A(\zeta_{\ell})[\ell] \neq 0$, it follows that $r_{\ell}\left(A(L_n)[\ell] \right) \geq 1$ for all $n \geq 1$. This gives us 
		\[ r_{\ell}\left( R_{\ell^{\infty}} (A/L_n) \right) \geq r_{\ell}\left(\text{Cl}_S(L_n) \right) - 2\dim(A). \]
		Applying Lemma~\ref{lem: S-class-group}, we get:
		\begin{align*}
			r_{\ell}\left( R_{\ell^{\infty}} (A/L_n) \right) 
			&\geq \left( r_{\ell}\left(\Cl(L_n)\right)  - 2s_0\ell^n \right)  - 2d  \\
			&\geq N + 2s_0p^n - 2s_0\ell^n  - 2\dim(A). 
		\end{align*}       
		Suppose $\ell < p$. Then we have
		\begin{align*}
			r_{\ell}\left( R_{\ell^{\infty}} (A/L_n) \right) 
			&\geq N + 2s_0\ell^n - 2s_0\ell^n  - 2d \\
			&= N\ell^n - 2\dim(A).
		\end{align*}
		Now suppose $\ell > p$. Then we have
		\begin{align*}
			r_{\ell}\left( R_{\ell^{\infty}} (A/L_n) \right) 
			&\geq N + 2s_0p^n - 2s_0p^n  - 2d \\
			&= Np^n - 2\dim(A).
		\end{align*}
		Either way, we get
		\[ r_{\ell}\left( R_{\ell^{\infty}} (A/L_n) \right) \geq Nq^n - 2\dim(A) \geq Nq^n,  \]
		where $q = \min\{\ell, p\}$.
		This completes the proof of Theorem~\ref{main-thm-gamma}.
	\end{proof}
	
	\section{Examples}
	
	\begin{example}
		Let $E=11a1$. Then $E(\Q(\zeta_5))[5] \neq 0$, so we can pick $\ell = 5$ and $p=3$. Let $N = 2$. We will construct a $\Gamma = \Z_3$-extension $L_{\infty}/L$ such that 
		\[ r_{5}(R_{{5}^{\infty}}(E/L_n)) \geq 2 \cdot 3^n \]
		for all $n \geq 0$. 
		
		We pick $K = \Q(\zeta_3)$ because there is a unique prime of $F$ lying over $p=3$. Let $K_{\infty}/K$ be the anticyclotomic $\Z_3$-extension of $K$. We have $t = N + 2d\ell (\ell-1) = 42$. We want to find $t=42$ primes $v_1, \dots, v_t$ of $F$ that split completely in $K_{\infty}$. It is enough to find $42$ primes different from $p=3$ that are inert in $\Q(\zeta_3)$. We list these primes below:
		\begin{align*}
			2, 5, 11, 17, 23, 29, 41, 47, 53, 59, 71, 83, 89, 101, 107, 113, 131, 137, \\ 149, 167, 173, 179, 191, 197, 227, 233, 239, 251, 257, 263, 269, 281, 293, 311, \\
			317, 347, 353, 359, 383, 389, 401, 419
		\end{align*}
		Let $\a$ be the product of these $42$ primes. We obtain:
		\begin{align*}
			\a = 
			55648213008781695672667810384702204705472968298668180461428\\
			1048399478905195501007583867510.
		\end{align*}
		Now put $L = K(\zeta_5, \sqrt[5]{\a})$. And let $L_{\infty} = K_{\infty} L$. Then $L_{\infty}/L$ is a $\Z_3$-extension. Theorem~\ref{thm-class-groups} says that 
		\[ r_{5}(\Cl(L_n)) \geq 2 \cdot 3^n  \]
		for all $n \geq 0$. And Theorem~\ref{main-thm-zp} says that 
		\[ r_5(R_{{5}^{\infty}}(E/L_n)) \geq 2 \cdot 3^n\]
		for all $n \geq 0$. 
	\end{example}

	\begin{example}
		We now look at $\Z_p^d$ extensions for $d=3$. As in the previous example, let $E=11a1$, $\ell = 5$ and $p=3$. Set $N = 10$. We will construct a $\Z_3^3$-extension $L_{\infty}/L$ such that 
		\[ r_{\ell}(R_{{\ell}^{\infty}}(E/L_n)) \geq 10 \cdot 3^n \]
		for all $n \geq 0$. 
		
		Consider  the CM field $K = \Q(\zeta_9)$. Let $K_{\infty}$ be the $\Z_3^3$-extension of $K$ given in Proposition~\ref{prop: part0-zpd-extensions}. We have $t = N + 2d\ell(\ell - 1) =90$. We want to find $t=90$ primes $v_1, \dots, v_t$ that are inert in $K$. By the well-known splitting laws for primes in cyclotomic fields, every rational prime which is $\equiv 2,5$ modulo $9$ is inert in $K$. Here is a list of $90$ such primes:
		\begin{align*}
			2, 5, 11, 23, 29, 41, 47, 59, 83, 101, 113, 131, 137, 149, 167, 173, 191, 227, 239, 257, 263, 281, 293,\\ 311, 317, 347, 353, 383, 389, 401, 419, 443, 461, 479, 491, 509, 563, 569, 587, 599, 617, 641, 653, 659,\\ 677, 743, 761, 797, 821, 839, 857, 887, 911, 929, 941, 947, 977, 983, 1013, 1019, 1031, 1049, 1091, 1103,\\ 1109, 1163, 1181, 1193, 1217, 1229, 1283, 1289, 1301, 1307, 1319, 1361, 1373, 1409, 1427, 1433,\\ 1451, 1481, 1487, 1499, 1523, 1553, 1559, 1571, 1607, 1613.
		\end{align*}
		
		Let $\a$ be the product of these primes:
		\begin{align*} 
			\a = 30266915671908567712011058723234542844654746560977147126408783068722197382\\3946203120683121105279988012699117394
			2884749088584144432870913089663861679\\
			02242957859532761609270923483095428112544069874627622945451584053107032901\\
			3191741865236750170.
		\end{align*}
		Now put $L = K(\zeta_5, \sqrt[5]{\a})$. And let $L_{\infty} = K_{\infty} L$. Then $L_{\infty}/L$ is a $\Z_3^2$-extension. 	Theorem~\ref{thm-class-groups} says that 
		\[ r_{5}(\Cl(L_n)) \geq 10 \cdot 3^n  \]
		for all $n \geq 0$. And Theorem~\ref{main-thm-gamma} says that 
		\[ r_5(R_{{5}^{\infty}}(E/L_n)) \geq 10 \cdot 3^n \]
		for all $n \geq 0$.  
	\end{example}

	\begin{example}
		We discuss a non-commutative example of $\Gamma$. The nilpotent uniform groups of dimension $d=3$ are parametrized, up to isomorphism, by a parameter $s \in \N$. They are given by (see \cite[Section 7, Theorem 7.4]{group-theory}) 
		\[ \Gamma(s) = \langle x,y,z: [x,z]=[y,z]=1, [x,y]=z^{p^s} \rangle. \]
		The groups $\Gamma(s)$ are non-abelian; they fit in the exact sequence
		\[ 1 \to \Z_p \to \Gamma(s) \to \Z_p^2 \to 1. \]
		If $p \equiv 1$ modulo $3$, the group $\Gamma(1)$ has an automorphism $\tau$ of order $3$ which has no fixed points (see \cite[Proposition 4.1]{hajir-maire}). Therefore $m=3$. 
		
		Put $\Gamma = \Gamma(1)$. Let $E=19a1$. Since $E(\Q)[3] \neq 0$, let $\ell = 3$, and $p=7$. Let $N = 6$. We will construct a $\Gamma$-extension $L_{\infty}/L$ such that 
		\[ r_{5}(R_{{5}^{\infty}}(E/L_n)) \geq 6 \cdot 3^n \]
		for all $n \geq 0$. 
		
		We construct a degree $m=3$ extension $F/F_0$ as in Assumption~\ref{assumption: m>2}. 	Let $F_0 = \Q(\zeta_7)$. Let $F = F_0(\theta)$, where $\theta$ is a root of the irreducible cubic polynomial 
		\[ x^3 - x^2 - 4x - 1. \]
		Then $F/F_0$ is a degree $3$ cyclic extension. According to the LMFDB database \cite{lmfdb-number-field1}, the class number of $F$ is $13$ and there is a unique prime $\mathfrak{p}$ lying over $p=7$. In particular, the $p$-part of the $\mathfrak{p}$-class group of $F$ is trivial so $F$ is $p$-rational. Therefore, $F/F_0$ satisfies the hypotheses of Proposition~\ref{prop: part1}.
		
		Set $t \coloneqq N + md\ell (\ell-1) = 6 + 3\cdot3\cdot3\cdot2 = 60$. We need to find $60$ primes of $F_0 = \Q(\zeta_7)$ which are inert in $F$. Consider the following $10$ rational primes. 
		\begin{align}
        \label{10-rational-primes}
			43, 127, 491, 673, 953, 1499, 1583, 2129, 2311, 2591.
		\end{align}		
		By a calculation in Sage, each of these rational primes splits completely in $F_0$ into $6$ factors and each of these factors is inert in $F$. For example, $43$ factors in $F_0$ as:
        \begin{align*}
            43 &= \left(\zeta_{7}^{5} + 2 \zeta_{7}^{3} + \zeta_{7}^{2} + 1\right) \cdot \left(\zeta_{7}^{5} + \zeta_{7}^{4} + 2 \zeta_{7}^{2} + \zeta_{7}\right) 
            \cdot \left(2 \zeta_{7}^{5} + \zeta_{7}^{4} + 2 \zeta_{7}^{3} + \zeta_{7}^{2} + 2 \zeta_{7} + 1\right) \\
            &\cdot \left(-2 \zeta_{7}^{5} - \zeta_{7}^{4} - \zeta_{7}^{3} - 2 \zeta_{7}^{2} - 2 \zeta_{7} - 1\right) 
            \cdot \left(2 \zeta_{7}^{4} + \zeta_{7}^{3} + \zeta_{7}^{2} + \zeta_{7}\right) 
            \cdot \left(\zeta_{7}^{5} + \zeta_{7}^{4} + \zeta_{7}^{3} + 2 \zeta_{7}^{2}\right),
        \end{align*}
        and each of the factors on the RHS is inert in $F$. In total, there are $10 \cdot 6 = 60$ primes of $F_0$ lying over the $10$ rational primes (\ref{10-rational-primes}), each of which is inert in $F$. 
		
		Let $\a$ be the product of these $60$ primes:
		\begin{align*}  
			\a = 78402503779216655405023576089116738265320606062683342998991230977\\
			29859436684020023921188941416161094578321474807227626638759156142079702\\
			108239313497652801991067685041337071171617321114788409671453358754013644971.
		\end{align*}
		We apply Lemma~\ref{lem: hajir-maire}. Since $[F_0: \Q] = 6 \geq 2d$, we set $n=0$ so that $K_0 = F_0$ and $K = F$. Lemma~\ref{lem: hajir-maire} says that there exists a $\Gamma$-extension $K_{\infty}/K$ such that every inert prime in $K/K_0$ splits completely in $K_{\infty}$. 
		
		Put $L = K(\a^{1/3}, \zeta_{3})$ and $L_{\infty} = LK_{\infty}$. Then $L_{\infty}/L$ is a $\Gamma$-extension. Theorem~\ref{thm-class-groups} says that 
		\[ r_{5}(\Cl(L_n)) \geq 6 \cdot 3^n  \]
		for all $n \geq 0$. And Theorem~\ref{main-thm-gamma} says that 
		\[ r_5(R_{{5}^{\infty}}(E/L_n)) \geq 6 \cdot 3^n\]
		for all $n \geq 0$.

	\end{example}

	\bibliographystyle{amsalpha}
	\bibliography{bibliography}
\end{document}